\documentclass[11pt]{article}%
\usepackage{amsmath}
\usepackage{amsfonts}
\usepackage{amssymb}
\usepackage{graphicx}%
\providecommand{\U}[1]{\protect\rule{.1in}{.1in}}
\newtheorem{theorem}{Theorem}

\newtheorem{remark}[theorem]{Remark}

\newenvironment{proof}[1][Proof]{\noindent\textbf{#1.} }{\ \rule{0.5em}{0.5em}}
\begin{document}

\title{A representation of the transmutation kernels for the Schr\"{o}dinger operator
in terms of eigenfunctions and applications}
\author{Kira V. Khmelnytskaya$^{1}$, Vladislav V. Kravchenko$^{2,3}$, Sergii M.
Torba$^{3}$\thanks{The authors acknowledge the support from CONACYT, Mexico
via the projects 284470 and 222478.}\\$^{1}${\footnotesize Faculty of Engineering, Autonomous University of
Queretaro, }\\{\footnotesize Cerro de las Campanas s/n, col.~Las Campanas Quer\'{e}taro,
Qro. C.P.~76010 M\'{e}xico}\\$^{2}${\small Regional mathematical center of Southern Federal University, }\\{\small Bolshaya Sadovaya, 105/42, Rostov-on-Don, 344006, Russia,}\\$^{3}${\footnotesize Department of Mathematics, Cinvestav, Unidad
Quer\'{e}taro}\\{\footnotesize Libramiento Norponiente \#2000, Fracc.~Real de Juriquilla,
 Quer\'{e}taro, Qro. C.P.~76230 M\'{e}xico}\\{\footnotesize khmel@uaq.edu.mx, vkravchenko@math.cinvestav.edu.mx,
storba@math.cinvestav.edu.mx}}
\maketitle

\begin{abstract}
The representations of the kernels of the transmutation operator and of its
inverse relating the one-dimensional Schr\"{o}dinger operator with the second
derivative are obtained in terms of the eigenfunctions of a corresponding
Sturm-Liouville problem. Since both series converge slowly and in general only
in a certain distributional sense we find a way to improve these expansions
and make them convergent uniformly and absolutely by adding and subtracting
corresponding terms. A numerical illustration of the obtained results is
given.

\end{abstract}

\section{Introduction}

Since the first work by J. Delsarte \cite{Delsarte1}, \cite{Delsarte2} the
transmutation operator relating the one-dimensional Schr\"{o}dinger operator
$L:=\frac{d^{2}}{dx^{2}}-q(x)$ to a more elementary operator $B:=\frac{d^{2}%
}{dx^{2}}$ has been subject of study  in hundreeds of publications devoted to
spectral theory and inverse problems (see, e.g., \cite{BegehrGilbert},
\cite{Carroll}, \cite{Katrakhov Sitnik 2018}, \cite{LevitanInverse},
\cite{Marchenko}, \cite{Yurko2007}). Recently in \cite{KNT2017},
\cite{Kr2018}, \cite{KTK2017} several representations of the integral kernel
of the transmutation operator in terms of series expansions in classical
orthogonal polynomials have been obtained equipped with convenient recurrent
formulas for the expansion coefficients. Every such representation leads to a
new functional series representation of the solutions of the Schr\"{o}dinger
equation which enjoys a remarkable uniformness property. It admits a spectral
parameter independent estimate for the approximation of the solution by
partial sums of the series, which in practice allows one to compute huge
numbers of eigenvalues and eigenfunctions with a controllable accuracy
\cite{KNT2017}, \cite{Kr2018}, \cite{KTK2017}. Similar results were obtained
for perturbed Bessel equations \cite{KTC2018}, \cite{KST2018}.

Up to now no similar representation has been obtained for the integral kernel
of the inverse transmutation operator which is required in numerous
applications, especially when solving initial-boundary value problems for PDEs
with variable coefficients. Moreover, an apparently unanswered  question is to
find an eigenfunction expansion of both the direct and the inverse
transmutation kernels, resembling the well known expansion of the Green
function. Such eigenfunction series expansions additionally to their profound
theoretical value acquire also computational significance due to the
availability of the representations of solutions admitting the spectral
parameter independent estimates and allowing one to compute huge amounts of eigendata.

In the present work we obtain an eigenfunction expansion of the integral
transmutation kernels of both the direct and the inverse transmutation
operators. Quite naturally, since the transmutation operators are related to
pairs of differential operators, the corresponding eigenfunction expansions
contain the eigendata of both differential operators.

The series expansions of both the direct and the inverse transmutation
operators converge slowly and in general only in a certain distributional
sense. We find the way to improve these expansions and make them convergent
uniformly and absolutely by adding and subtracting corresponding terms. We
give a numerical illustration of the obtained results.

\section{Preliminaries}

Let $q$ be a\ real valued function belonging to $L_{2}(0,\pi)$ and $h$, $H$ be
two real numbers. Consider the Sturm-Liouville problem
\begin{gather}
-y^{\prime\prime}+q(x)y=\lambda y,\label{SL1}\\
y^{\prime}(0)-hy(0)=y^{\prime}(\pi)+Hy(\pi)=0.\label{boundary conditions}%
\end{gather}
It defines two sequences of real numbers, the eigenvalues $\left\{
\lambda_{n}\right\}  _{n=0}^{\infty}$ and the weight numbers (or normalizing
constants) $\left\{  \alpha_{n}\right\}  _{n=0}^{\infty}$ such that
$\lambda_{n}\neq\lambda_{m}$ for $n\neq m$, $\alpha_{n}>0$,
\begin{equation}
\rho_{n}:=\sqrt{\lambda_{n}}=n+\frac{\omega}{\pi n}+\frac{k_{n}}{n}%
,\qquad\alpha_{n}=\frac{\pi}{2}+\frac{K_{n}}{n},\qquad\left\{  k_{n}\right\}
_{n=0}^{\infty},\left\{  K_{n}\right\}  _{n=0}^{\infty}\in l_{2}%
.\label{asymptotics of spectral data}%
\end{equation}
The weight numbers are defined as follows $\alpha_{n}:=\int_{0}^{\pi}%
c^{2}(\rho_{n},x)dx$ where $c(\rho,x)$ denotes the solution of the Cauchy
problem%
\begin{gather*}
-c^{\prime\prime}(\rho,x)+q(x)c(\rho,x)=\rho^{2}c(\rho,x),\\
c(\rho,0)=1,\qquad c^{\prime}(\rho,0)=h.
\end{gather*}
The value of the number $\omega$ in (\ref{asymptotics of spectral data}) is
given by the formula%
\begin{equation}
\omega=h+H+\frac{1}{2}\int_{0}^{\pi}q(t)\,dt.\label{omega}%
\end{equation}
Consider the Gel'fand-Levitan equation
\begin{equation}
G(x,t)+F(x,t)+\int_{0}^{x}F(t,s)G(x,s)\,ds=0,\qquad0\leq
t<x,\label{Gelfand-Levitan}%
\end{equation}
where $F(x,t)$ has the form%
\begin{equation}
F(x,t)=\sum_{n=0}^{\infty}\left(  \frac{\cos\rho_{n}x\cos\rho_{n}t}{\alpha
_{n}}-\frac{\cos nx\cos nt}{\alpha_{n}^{0}}\right)  \label{F}%
\end{equation}
with
\[
\alpha_{n}^{0}=%
\begin{cases}
\pi/2, & n>0,\\
\pi, & n=0
\end{cases}
\]
and $G(x,t)$ is the kernel of a transmutation operator
\begin{equation}
T[u](x):=u(x)+\int_{0}^{x}G(x,s)u(s)\,ds\label{Ttransmute}%
\end{equation}
relating the operator $L:=\frac{d^{2}}{dx^{2}}-q(x)$ with the operator
$B:=\frac{d^{2}}{dx^{2}}$ as follows. Let $u\in C^{2}[0,\pi]$ and $u^{\prime
}(0)=0$. Then $LTu=TBu$. Denote $v:=Tu$. Then $v^{\prime}(0)-hv(0)=0$. In
particular,
\begin{equation}
T[\cos\rho x]=c(\rho,x).\label{Tcos}%
\end{equation}

Let $H(x,t)$ denote the kernel of the inverse transmutation operator. That is
\begin{equation}
\label{Tinverse}T^{-1}\left[  u\right]  \left(  x\right)  =u\left(  x\right)
+\int_{0}^{x}H(x,s)u(s)\,ds.
\end{equation}
Recall that the kernel $H(x,t)$ can be obtained from the equality \cite[Lemma
1.3.9]{Yurko2007},%
\begin{equation}
H(x,t)-F(x,t)-\int_{0}^{t}F(x,s)G(t,s)\,ds=0,\quad0\leq t<x.
\label{H via F and G}%
\end{equation}

The integral kernels $G$ and $H$ are continuous functions in $0\leq t\leq
x\leq\pi$ and satisfy the equalities
\begin{equation}
G(x,x)=-H(x,x)=h+\frac{1}{2}\int_{0}^{x}q(t)\,dt.\label{GHcaract}%
\end{equation}

\section{Series representations for the transmutation kernels}

\begin{remark}
Equations \eqref{Gelfand-Levitan} and \eqref{H via F and G} can be written in
the form
\begin{equation}
G(x,t)=-T_{x}[ F] (x,t) \label{G-L1}%
\end{equation}
and
\begin{equation}
H(x,t)=T_{t}[ F] (x,t), \label{H1}%
\end{equation}
respectively. That is, the kernels $G(x,t)$ and $H(x,t)$ are images of the
function $\mp F(x,t)$ under the action of the transmutation operator applied
with respect to the variable $x$ and $t$, respectively.
\end{remark}

\begin{theorem}
\label{Th Representation of kernels}The kernels $G(x,t)$ and $H(x,t)$ admit
the following representations%
\begin{equation}
G(x,t)=\sum_{n=0}^{\infty}\left(  \frac{c\left(  n,x\right)  \cos nt}%
{\alpha_{n}^{0}}-\frac{c\left(  \rho_{n},x\right)  \cos\rho_{n}t}{\alpha_{n}%
}\right)  \label{G}%
\end{equation}
and
\begin{equation}
H(x,t)=\sum_{n=0}^{\infty}\left(  \frac{\cos\rho_{n}x\,c\left(  \rho
_{n},t\right)  }{\alpha_{n}}-\frac{\cos nx\,c\left(  n,t\right)  }{\alpha
_{n}^{0}}\right)  ,\label{H}%
\end{equation}
where the series converge in the following distributional sense. Let the
integral kernels $G(x,t)$ and $H(x,t)$ be extended by zero for $t>x$. Then for
any $f\in AC\left[  0,\pi\right]  $ the following limit (corresponding to
\eqref{G} and \eqref{H}) exists uniformly with respect to $x$,
\begin{equation}%
\begin{split}
\lim_{N\rightarrow\infty}\int_{0}^{\pi}f(t)I_{N}(x,t)dt &  =\int_{0}^{\pi
}f(t)\bigl(G(x,t)-H(t,x)\bigr)\,dt\\
&  =-\int_{0}^{\pi}f(t)\left(  F(x,t)+\int_{0}^{x}G(x,s)F(s,t)\,ds\right)  dt,
\end{split}
\label{DistributionalSense}%
\end{equation}
where
\[
I_{N}(x,t)=\sum_{n=0}^{N}\left(  \frac{c\left(  n,x\right)  \cos nt}%
{\alpha_{n}^{0}}-\frac{c\left(  \rho_{n},x\right)  \cos\rho_{n}t}{\alpha_{n}%
}\right)  .
\]
Note that for $x\neq t$ only one of the values $G(x,t)$ or $H(t,x)$ can be
different from zero.
\end{theorem}

\begin{remark}
\label{Remark Discontinuity} As can be concluded from results of Section
\ref{Section Improvement} and Jordan's theorem \cite[Chap.1, \S 39]{Bary1964},
the series in \eqref{G} and \eqref{H} converge pointwise on $[0,\pi
]\times\lbrack0,\pi]$, uniformly on any compact subset of $[0,\pi
)\times\lbrack0,\pi)$ and their sums are continuous functions except a jump
discontinuity at $x=t=\pi$, the size of the jump being $\omega$, where $\omega$ is the parameter from \eqref{omega}.
\end{remark}

\begin{proof}
Consider
\[
T_{x}\left[  \frac{\cos\rho_{n}x\cos\rho_{n}t}{\alpha_{n}}-\frac{\cos nx\cos
nt}{\alpha_{n}^{0}}\right]  =\frac{c\left(  \rho_{n},x\right)  \cos\rho_{n}%
t}{\alpha_{n}}-\frac{c\left(  n,x\right)  \cos nt}{\alpha_{n}^{0}}.
\]
Hence, formally,
\[
G(x,t)=-T_{x}\sum_{n=0}^{\infty}\left(  \frac{\cos\rho_{n}x\cos\rho_{n}%
t}{\alpha_{n}}-\frac{\cos nx\cos nt}{\alpha_{n}^{0}}\right)  =\sum
_{n=0}^{\infty}\left(  \frac{c\left(  n,x\right)  \cos nt}{\alpha_{n}^{0}%
}-\frac{c\left(  \rho_{n},x\right)  \cos\rho_{n}t}{\alpha_{n}}\right)  ,
\]
and similarly for $H$.

Due to \eqref{Ttransmute}, \eqref{Tinverse} and \eqref{Tcos} one has
\[
c(n,x)=\cos nx+\int_{0}^{x}G(x,s)\cos ns\,ds,\qquad\cos\rho_{n}x=c(\rho
_{n},x)+\int_{0}^{x}H(x,s)c(\rho_{n},s)\,ds,
\]
hence
\[%
\begin{split}
I_{N}(x,t): &  =\sum_{n=0}^{N}\left(  \frac{c(n,x)\cos nt}{\alpha_{n}^{0}%
}-\frac{c(\rho_{n},x)\cos\rho_{n}t}{\alpha_{n}}\right)  \\
&  =\sum_{n=0}^{N}\left(  \frac{\cos nx\cos nt}{\alpha_{n}^{0}}-\frac{\cos
\rho_{n}x\cos\rho_{n}t}{\alpha_{n}}\right)  \\
&  \quad+\sum_{n=0}^{N}\frac{\cos nt}{\alpha_{n}^{0}}\int_{0}^{x}G(x,s)\cos
ns\,ds-\sum_{n=0}^{N}\frac{\cos\rho_{n}t}{\alpha_{n}}\int_{0}^{x}%
H(x,s)\cos\rho_{n}s\,ds\\
&  =:\Phi_{N}(x,t)+I_{N,2}(x,t)+I_{N,4}(x,t)
\end{split}
\]
(here we follow notation from the proof of Theorem 1.3.1 \cite{Yurko2007}).

According to \cite[proof of Theorem 1.3.1]{Yurko2007} for any $f\in AC\left[
0,\pi\right]  $ the following limit exists
\[
\lim_{N\rightarrow\infty}\max_{0\leq x\leq\pi}\int_{0}^{\pi}f(t)\Phi
_{N}(x,t)\,dt=0.
\]
Moreover, uniformly with respect to $x\in\lbrack0,\pi]$,
\begin{align*}
\lim_{N\rightarrow\infty}\int_{0}^{\pi}f(t)I_{N,2}(x,t)dt &  =\int_{0}%
^{x}f(t)G(x,t)dt,\\
\lim_{N\rightarrow\infty}\int_{0}^{\pi}f(t)I_{N,4}(x,t)dt &  =-\int_{x}^{\pi
}f(t)H(t,x)dt.
\end{align*}
Hence, extending $G(x,t)$ and $H(x,t)$ by zero for $x<t$ we obtain that
\[
\lim_{N\rightarrow\infty}I_{N}(x,t)=\int_{0}^{\pi}%
f(t)\bigl(G(x,t)-H(t,x)\bigr)dt,
\]
thus establishing \eqref{DistributionalSense}.
\end{proof}

\begin{remark}
Our prime interest consists in application of the representation \eqref{H} for
computing the preimages of functions under the action of the transmutation
operator. From Theorem \ref{Th Representation of kernels} we have that for any
$f\in AC\left[  0,\pi\right]  $ the sequence
\begin{multline}
f(x)+\sum_{n=0}^{N}\left(  \frac{\cos\rho_{n}x\,}{\alpha_{n}}\int_{0}%
^{x}c(\rho_{n},t)f(t)\,dt-\frac{\cos nx\,}{\alpha_{n}^{0}}\int_{0}%
^{x}c(n,t)f(t)\,dt\right)  \label{preimageN}\\
=f(x)+\int_{0}^{x}\sum_{n=0}^{N}\left(  \frac{\cos\rho_{n}x\,c(\rho_{n}%
,t)}{\alpha_{n}}-\frac{\cos nx\,c(n,t)}{\alpha_{n}^{0}}\right)  f(t)\,dt
\end{multline}
tends to $T^{-1}\left[  f\right]  \left(  x\right)  $ uniformly. This gives us
a practical way for computing the preimages of absolutely continuous functions
reducing such computation to a number of definite integrals from \eqref{preimageN}.
\end{remark}

\begin{remark}
The convergence rate of the series in the representations \eqref{F}, \eqref{G}
and \eqref{H} improves when the parameter $\omega$ in
\eqref{asymptotics of spectral data} equals zero, see Section
\ref{Section Improvement} for details. Note that by appropriate choice of the
constant $H$ in \eqref{boundary conditions} the parameter $\omega$ can always
be set to zero. Since the kernels $G(x,t)$ and $H(x,t)$ do not depend on $H$,
a right choice of the constant $H$ can lead to a faster convergence of the
series in \eqref{F}, \eqref{G} and \eqref{H}. In the next section we show how
the series \eqref{G} and \eqref{H} can be modified in order to improve the
convergence rate even when $\omega\neq0$.
\end{remark}

\section{A Fourier series of a discontinuous function and improvement of
convergence}

\label{Section Improvement} Both series \eqref{G} and \eqref{H} converge
rather slowly, and, moreover, for $x=t=\pi$ they converge to the exact values
$G(\pi,\pi)$ and $H(\pi,\pi)$  if only the number $\omega$ from \eqref{omega}
equals zero.

A simple explanation (and an idea how to improve the convergence for an
arbitrary value of $\omega$) can be seen in the proof of Lemma 1.3.4 from
\cite{Yurko2007}. We briefly repeat the formulas for the function $F$ and
later present the proof for the functions $G$ and $H$. Following
\cite{Yurko2007} let us introduce the following function
\[
a(x)=\sum_{n=0}^{\infty}\left(  \frac{\cos\rho_{n}x}{\alpha_{n}}-\frac{\cos
nx}{\alpha_{n}^{0}}\right)  .
\]
The function $F$ can be expressed in the terms of the function $a$ as
follows:
\begin{equation}
F(x,t)=\frac{1}{2}\bigl(a(x+t)+a(x-t)\bigr).\label{F sum as}%
\end{equation}
Note that the function $F$ should be defined on the region $(x,t)\in
\lbrack0,\pi]\times\lbrack0,\pi]$ to be able to consider equation
\eqref{Gelfand-Levitan}, hence the function $a$ should be defined for
$x\in\lbrack-\pi,2\pi]$.

The function $a(x)$ can be represented as (see the proof of Lemma 1.3.4
\cite{Yurko2007} and notice that the factor $1/\alpha_{n}^{0}=2/\pi$ is
missing in the proof)
\begin{equation}
a(x)=-\frac{2\omega x}{\pi^{2}}\sum_{n=1}^{\infty}\frac{\sin nx}{n}%
+A_{2}(x),\label{a(x)Yurko}%
\end{equation}
where the function $A_{2}(x)$ is continuous on $[-\pi,2\pi]$. As for the first
sum in \eqref{a(x)Yurko}, we have
\begin{equation}
-\frac{2\omega x}{\pi^{2}}\sum_{n=1}^{\infty}\frac{\sin nx}{n}=%
\begin{cases}
-\frac{\omega|x|(\pi-|x|)}{\pi^{2}}, & |x|<2\pi,\\
0, & x=2\pi.
\end{cases}
\label{a(x)problem}%
\end{equation}
One can easily see that whenever $\omega\neq0$, there is a jump discontinuity
at $x=2\pi$, corresponding to $x=t=\pi$ for the function $F(x,t)$. Moreover,
the series \eqref{a(x)problem} converges slowly which implies the slow
convergence of the series representing the function $F$. Same happens to the
functions $G(x,t)$ and $H(x,t)$.

The idea to improve the convergence is to consider the following expression
for the function $a(x)$:
\begin{equation}
a(x)=\frac{\cos\rho_{0}x}{\alpha_{0}}-\frac{1}{\pi}+\sum_{n=1}^{\infty}\left(
\frac{\cos\rho_{n}x}{\alpha_{n}}-\frac{\cos nx}{\alpha_{n}^{0}}+\frac{2\omega
x}{\pi^{2}}\frac{\sin nx}{n}\right)  -\frac{\omega|x|(\pi-|x|)}{\pi^{2}%
},\label{a(x)correction}%
\end{equation}
that is, to subtract the slowly convergent series termwise and to add the
closed expression for the whole infinite sum. Note that as an additional
benefit of such reformulation, the series in \eqref{a(x)correction} is
uniformly convergent on the whole segment $[-\pi,2\pi]$, and hence the
function $a$ given by \eqref{a(x)correction} is continuous on $[-\pi,2\pi]$.
Based on this idea the following result is obtained.

\begin{theorem}
\label{Thm correction} The kernels $G$ and $H$ admit the following
representations
\begin{equation}%
\begin{split}
G(x,t) & =\sum_{n=1}^{\infty}\left(  \frac{c( n,x) \cos nt}{\alpha_{n}^{0}}-
\frac{c( \rho_{n},x) \cos\rho_{n}t}{\alpha_{n}} - \frac{2\omega}{\pi^{2}%
n}\Bigl(x\sin nx \cos nt + t\sin nt\cos nx\Bigr) \right) \\
& \quad+ \frac{c( 0,x)}{\pi}- \frac{c( \rho_{0},x) \cos\rho_{0}t}{\alpha_{0}}+
\frac{\omega}{\pi^{2}}\bigl(\pi x -x^{2}-t^{2}\bigr),\qquad0\le t\le x\le\pi
\end{split}
\label{Gcorrected}%
\end{equation}
and
\begin{equation}%
\begin{split}
H(x,t) & =\sum_{n=1}^{\infty}\left( \frac{\cos\rho_{n}x\,c(\rho_{n},t)}%
{\alpha_{n}}- \frac{\cos nx\,c(n,t)}{\alpha_{n}^{0}} + \frac{2\omega}{\pi
^{2}n}\Bigl(x\sin nx \cos nt + t\sin nt\cos nx\Bigr)\right) \\
& \quad+\frac{\cos\rho_{0}x\,c(\rho_{0},t)}{\alpha_{0}}- \frac{c(0,t)}{\pi
}-\frac{\omega}{\pi^{2}}\bigl(\pi x -x^{2}-t^{2}\bigr), \qquad0\le t\le
x\le\pi
\end{split}
\label{Hcorrected}%
\end{equation}
where the series converge uniformly and absolutely.
\end{theorem}

\begin{proof}
Let us show that the series for $F(x,t)$ and $H(x,t)$ given by \eqref{F} and
\eqref{H} differ from each other by an absolutely and uniformly convergent
series of continuous functions. Consider
\[%
\begin{split}
\Delta_{n} &  =\left(  \frac{\cos\rho_{n}x\,c(\rho_{n},t)}{\alpha_{n}}%
-\frac{\cos nx\,c(n,t)}{\alpha_{n}^{0}}\right)  -\left(  \frac{\cos\rho
_{n}x\,\cos\rho_{n}t}{\alpha_{n}}-\frac{\cos nx\,\cos nt}{\alpha_{n}^{0}%
}\right)  \\
&  =\frac{\cos\rho_{n}x}{\alpha_{n}}\bigl(c(\rho_{n},t)-\cos\rho
_{n}t\bigr)-\frac{\cos nx}{\alpha_{n}^{0}}\bigl(c(n,t)-\cos nt\bigr).
\end{split}
\]

The function $c(\rho,t)$ satisfies the following asymptotic relation
\cite[(1.1.15)]{Yurko2007}
\[
c(\rho,t)=\cos\rho t+q_{1}(t)\frac{\sin\rho t}{\rho}+\int_{0}^{t}%
q(s)\frac{\sin\rho(t-2s)}{2\rho}ds+O\left(  \frac{1}{\rho^{2}}\right)
,\qquad\rho\rightarrow\infty,
\]
where $q_{1}(t)=h+\frac{1}{2}\int_{0}^{t}q(s)\,ds$. Hence
\begin{equation}%
\begin{split}
\Delta_{n} &  =q_{1}(t)\left(  \frac{\cos\rho_{n}x\,\sin\rho_{n}t}{\alpha
_{n}\rho_{n}}-\frac{\cos nx\,\sin nt}{n\alpha_{n}^{0}}\right)  \\
&  \quad+\int_{0}^{t}q(t)\left(  \frac{\cos\rho_{n}x\,\sin\rho_{n}%
(t-2s)}{2\alpha_{n}\rho_{n}}-\frac{\cos nx\,\sin n(t-2s)}{2n\alpha_{n}^{0}%
}\right)  ds+O\left(  \frac{1}{n^{2}}\right)  ,\qquad n\rightarrow\infty.
\end{split}
\label{Expr for Delta n}%
\end{equation}
We obtain for the first term that
\[%
\begin{split}
\Delta_{n}^{1} &  :=\frac{\cos\rho_{n}x\,\sin\rho_{n}t}{\alpha_{n}\rho_{n}%
}-\frac{\cos nx\,\sin nt}{n\alpha_{n}^{0}}\\
&  =\frac{\cos\rho_{n}x\,\sin\rho_{n}t-\cos nx\,\sin nt}{\alpha_{n}\rho_{n}%
}+\left(  \frac{1}{\alpha_{n}\rho_{n}}-\frac{1}{n\alpha_{n}^{0}}\right)  \cos
nx\,\sin nt\\
&  =\frac{\sin\frac{(\rho_{n}-n)(x+t)}{2}\cos\frac{(\rho_{n}+n)(x+t)}{2}%
-\sin\frac{(\rho_{n}-n)(x-t)}{2}\cos\frac{(\rho_{n}+n)(x-t)}{2}}{\alpha
_{n}\rho_{n}}-\frac{\alpha_{n}\rho_{n}-n\alpha_{n}^{0}}{n\alpha_{n}\alpha
_{n}^{0}\rho_{n}}\cos nx\,\sin nt.
\end{split}
\]
Taking into account \eqref{asymptotics of spectral data} one can see that
$\rho_{n}-n=O(1/n)$ and $\alpha_{n}\rho_{n}-n\alpha_{n}^{0}=O(1)$, hence
$\Delta_{n}^{1}=O(1/n^{2})$, $n\rightarrow\infty$. Similarly for the second
term in \eqref{Expr for Delta n}.

Now, combining \eqref{F sum as} with \eqref{a(x)correction} we obtain that
\begin{equation}
\label{F corrected}%
\begin{split}
F(x,t)  &  = \sum_{n=1}^{\infty}\left( \frac{\cos\rho_{n} x \, \cos\rho_{n}%
t}{\alpha_{n}} - \frac{\cos nx\, \cos nt}{\alpha_{n}^{0}}+\frac{\omega
(x+t)}{\pi^{2}}\frac{\sin n(x+t)}{n} +\frac{\omega(x-t)}{\pi^{2}}\frac{\sin
n(x-t)}{n}\right) \\
&  \quad+\frac{\cos\rho_{0} x \, \cos\rho_{0} t}{\alpha_{0}} - \frac{1}{\pi}-
\frac{\omega|x+t|(\pi-|x+t|)}{2\pi^{2}}-\frac{\omega|x-t|(\pi-|x-t|)}{2\pi
^{2}}\\
&  = \sum_{n=1}^{\infty}\left( \frac{\cos\rho_{n} x \, \cos\rho_{n}t}%
{\alpha_{n}} - \frac{\cos nx\, \cos nt}{\alpha_{n}^{0}}+\frac{2\omega}{\pi^{2}
n}\Bigl( x\sin nx\,\cos nt + t\sin nt\,\cos nx\Bigr)\right) \\
& \quad+\frac{\cos\rho_{0} x \, \cos\rho_{0} t}{\alpha_{0}} - \frac{1}{\pi} -
\frac{\omega}{\pi^{2}}\bigl(\pi\max\{x,t\} -x^{2}-t^{2}\bigr),
\end{split}
\end{equation}
where the series converges absolutely and uniformly and the equality holds
whenever $x<\pi$ or $t<\pi$, i.e., $\min\{x,t\}<\pi$.

Since
\begin{equation}
\label{Diff H and F}H(x,t) = F(x,t) - \sum_{n=0}^{\infty}\Delta_{n},
\end{equation}
where the series $\sum_{n=0}^{\infty}\Delta_{n}$ converges absolutely and
uniformly, we immediately obtain \eqref{Hcorrected} from \eqref{F corrected}
and \eqref{Diff H and F} for all $x,t$ which satisfy $\min\{x,t\}<\pi$. Note
also that the kernel $H$ is a continuous function on $0\le t\le x\le\pi$, and
the right hand side of \eqref{Hcorrected} is also a continuous function on the
same region. Hence the equality holds for $x=t=\pi$ as well.

The proof of \eqref{Gcorrected} is completely similar.
\end{proof}

\section{An explicit example}

Let us consider an exactly solvable example which reveals some important
features of the representations (\ref{G}) and (\ref{H}). Let $q\equiv1$ and
$h=0$. Then
\[
c(\rho,x) =T[ \cos\rho x] =\cos\mu x
\]
where $\mu:=\sqrt{\rho^{2}-1}$. In particular,
\[
c( 0,x) =T[ 1] =\cosh x.
\]

The integral kernel $G$ for this example is given by, see \cite[Example 6]{KT
Transmut} and \cite[(1.2.7)]{Marchenko}
\begin{equation}
\label{Gexact}G(x,t) =
\begin{cases}
\frac{x I_{1}(\sqrt{x^{2}-t^{2}})}{\sqrt{x^{2}-t^{2}}}, & t<x,\\
\frac x2, & t=x,
\end{cases}
\end{equation}
where $I_{1}$ is the modified Bessel function of the first kind.

Consider first the corresponding Sturm-Liouville problem with $H=0$. Then
$c\left(  \rho_{n},x\right)  =\cos\sqrt{\rho_{n}^{2}-1}x$, $\rho_{n}%
=\sqrt{n^{2}+1}$, $\alpha_{n}=\alpha_{n}^{0}$. Hence \eqref{G} gives
\begin{equation}%
\begin{split}
\label{G 1}G_{1}(x,t) & =\sum_{n=0}^{\infty}\frac{1}{\alpha_{n}^{0}}\left(
\cos\sqrt{n^{2}-1}x\,\cos nt-\cos nx\,\cos\sqrt{n^{2}+1}t\right) \\
& = \frac{\cosh x - \cos t}\pi+ \sum_{n=1}^{\infty}\frac{2}{\pi}\left(
\cos\sqrt{n^{2}-1}x\,\cos nt-\cos nx\,\cos\sqrt{n^{2}+1}t\right) .
\end{split}
\end{equation}
On the other hand choosing the constant $H$ in (\ref{boundary conditions})
equal to $-\pi/2$ we obtain that $\omega=0$. Let us construct the
corresponding series representation for $G(x,t)$. Thus, $q\equiv1$, $h=0$ and
$H=-\pi/2$. Then $c( \rho_{n},x) =\cos\mu_{n}x$, where $\mu_{n}=\sqrt{\rho
_{n}^{2}-1}$ are solutions of the characteristic equation
\begin{equation}
\mu\sin\mu\pi+\frac{\pi}{2}\cos\mu\pi=0 \label{character}%
\end{equation}
and $\alpha_{n}=\frac{\pi}{2}+\frac{\sin2\mu_{n}\pi}{4\mu_{n}}$. Notice that
the problem possesses one negative eigenvalue $\rho_{0}^{2}\approx-1.468$. The
corresponding series representation for $G(x,t)$ takes the form%
\begin{equation}%
\begin{split}
\label{G 2}G_{2}(x,t)  &  =\frac{\cosh x}{\pi}-\frac{\cos\mu_{0}x\,\cos
\rho_{0}t}{\alpha_{0}}\\
&  \quad+\sum_{n=1}^{\infty}\left(  \frac{2\cos\sqrt{n^{2}-1}x\,\cos nt}{\pi
}-\frac{\cos\mu_{n}x\,\cos\rho_{n}t}{\alpha_{n}}\right)  .
\end{split}
\end{equation}

Additionally, for the case $h=H=0$ we consider the representation given by
\eqref{Gcorrected},
\begin{equation}%
\begin{split}
\label{G 3}
G_{3}(x,t) &  =\sum_{n=1}^{\infty}\frac{2}{\pi}\left(  \cos\sqrt{n^{2}%
-1}x\,\cos nt-\cos nx\,\cos\sqrt{n^{2}+1}t-\frac{x\sin nx\,\cos nt+t\sin
nt\,\cos nx}{2n}\right)  \\
&  \quad+\frac{\cosh x-\cos t}{\pi}+\frac{\omega}{\pi^{2}}\bigl(\pi
x-x^{2}-t^{2}\bigr).
\end{split}
\end{equation}

We stress that the representations \eqref{G 1}, \eqref{G 2} and \eqref{G 3}
correspond to the same kernel $G(x,t)$.

We computed approximate integral kernels by truncating the series in
\eqref{G 1}, \eqref{G 2} and \eqref{G 3} and compared with the exact integral
kernel at $x=\pi$. On Figure \ref{Figure KernelErr} we present the absolute
value of the differences, 10 and 100 terms of the series were used. In
accordance with Remark \ref{Remark Discontinuity} the difference between
partial sums of the series \eqref{G 1} and the exact value at $t=\pi$ remains
close to $\pi/2$, while both series \eqref{G 2} and \eqref{G 3} converge
uniformly and faster. All computations were realized in Matlab 2017a. Notice
that opposite to \eqref{G 1} and \eqref{G 3}, the approximation obtained from
\eqref{G 2} requires $\mu_{n}$ to be computed numerically from
\eqref{character}. This was done by converting the function $\mu\sin\mu
\pi+\frac{\pi}{2}\cos\mu\pi$ into a spline and finding its zeros with the aid
of the Matlab routine \texttt{fnzeros}.

\begin{figure}[tbh]
\centering
\includegraphics[bb=0 0 216 180, width=3in,height=2.5in]
{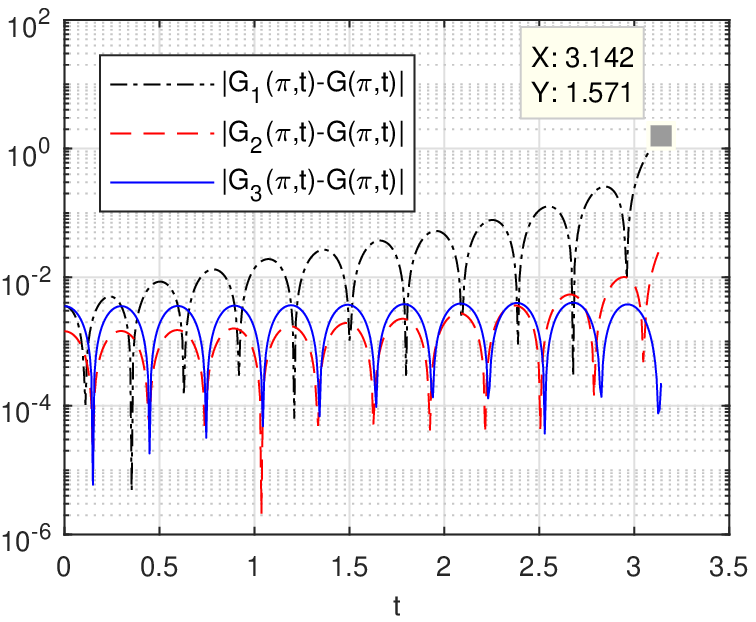}\ \ \ \includegraphics[bb=0 0 216 180, width=3in,height=2.5in]
{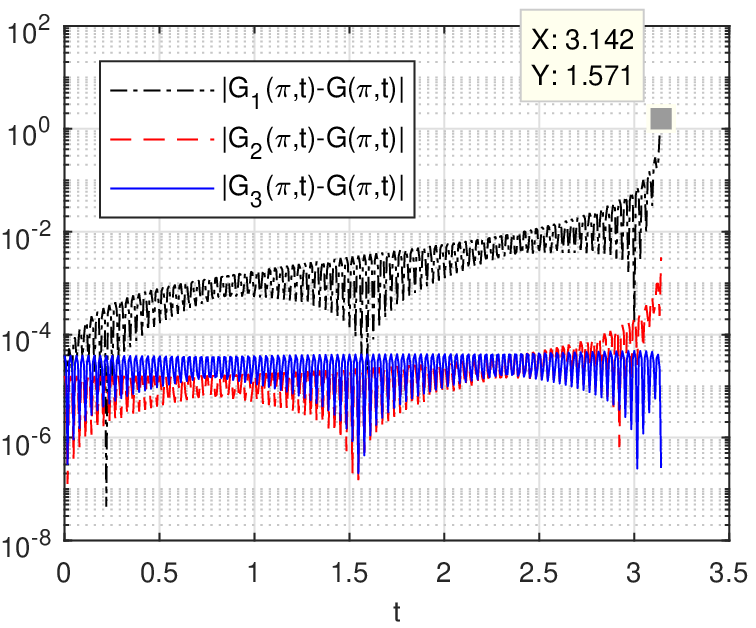}\caption{Absolute errors of the approximate integral kernel
computed using truncated sums of the series \eqref{G 1}, \eqref{G 2} and
\eqref{G 3}. Left plot: 10 terms used. Right plot: 100 terms used.}%
\label{Figure KernelErr}%
\end{figure}

Now let us compare the convergence rate of the series applying all three
representations for computing $T[1]$. Of course, in our example,
$T[1]=c\left(  0,x\right)  =\cosh x$. Thus, using \eqref{G 1}, \eqref{G 2} and
\eqref{G 3} we construct three approximations of the function $\cosh x$,
\begin{align}
\cosh x  &  \approx1+\frac{x\cosh x-\sin x}{\pi}+\sum_{n=1}^{N}\frac{2}{\pi
}\left(  \frac{\cos\sqrt{n^{2}-1}x\,\sin nx}{n}-\frac{\cos nx\,\sin\sqrt
{n^{2}+1}x}{\sqrt{n^{2}+1}}\right) ,\label{G1 approx}\\
\cosh x  &  \approx1+\frac{x\cosh x}{\pi}-\frac{\cos\mu_{0}x\,\sin\rho_{0}%
x}{\alpha_{0}\rho_{0}} +\sum_{n=1}^{N}\left(  \frac{2\cos\sqrt{n^{2}-1}x\,\sin
nx}{\pi n}-\frac{\cos\mu_{n}x\,\sin\rho_{n}x}{\alpha_{n}\rho_{n}}\right)
,\label{G2 approx}%
\end{align}
and
\begin{equation}
\label{G3 approx}%
\begin{split}
\cosh x  &  \approx1+\frac{x\cosh x-\sin x}{\pi} +\frac{\omega x^{2}}{\pi} -
\frac{4\omega x^{3}}{3\pi^{2}}\\
&  \quad+\sum_{n=1}^{N}\frac{2}{\pi}\left(  \frac{\cos\sqrt{n^{2}-1}x\,\sin
nx}{n}-\frac{\cos nx\,\sin\sqrt{n^{2}+1}x}{\sqrt{n^{2}+1}} + \frac{x\cos
2nx}{2n^{2}} - \frac{\sin2nx}{4n^{3}}\right) ,
\end{split}
\end{equation}
respectively.

For $N=10$ the absolute error of the first approximation was $9.5\cdot10^{-2}%
$, of the second $5.5\cdot10^{-4}$ and of the third $3.3\cdot10^{-4}$. For
$N=100$ the absolute error of the first approximation was $9.9\cdot10^{-3}$,
of the second $4.0\cdot10^{-6}$ and that of the third $3.8\cdot10^{-7}$.
Finally, for $N=1000$ the absolute error of the first approximation was
$1.0\cdot10^{-3}$, of the second $9.8\cdot10^{-9}$ and that of the third
$3.9\cdot10^{-10}$. All three series converge slowly. However the convergence
rate greatly improves either considering the second representation
\eqref{G2 approx} corresponding to $\omega=0$ or the third representation
\eqref{G3 approx}. On Figure \ref{Figure T1} we present the absolute errors of
representations \eqref{G1 approx}, \eqref{G2 approx} and \eqref{G3 approx} as
functions of $N$.

\begin{figure}[tbh]
\centering
\includegraphics[bb=0 0 324 180, width=4.5in,height=2.5in]
{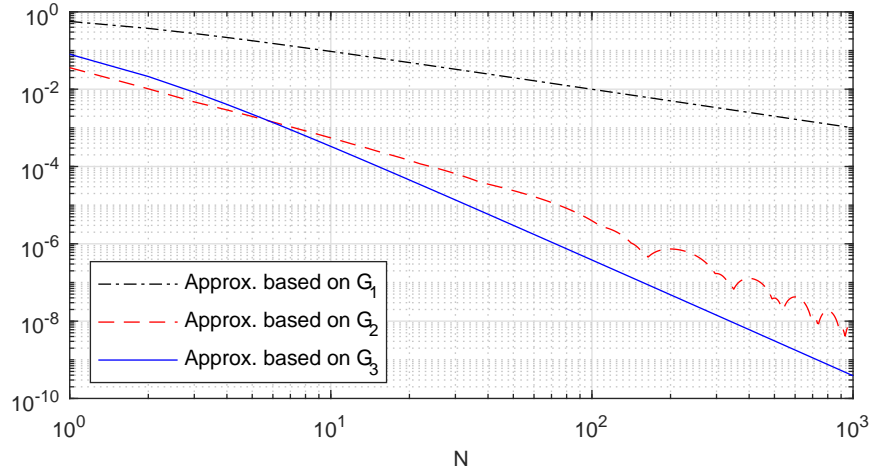}\caption{Absolute errors of computation of $T[1]$ using series
\eqref{G1 approx} (black dot-dashed line), \eqref{G2 approx} (red dashed line)
and \eqref{G3 approx} (solid blue line) as functions of the number of terms
used.}%
\label{Figure T1}%
\end{figure}

\end{document}